\documentclass [11pt,twoside,a4paper]{article}
\usepackage{amsfonts}
\usepackage{amsthm}
\usepackage{amsmath}
\usepackage{amscd}
\usepackage{amsbsy}            %
\usepackage{psfrag}            %
\usepackage{epsf}              
\usepackage{graphicx}          %
\usepackage{makeidx}           %
\usepackage{color}             %
\usepackage{fancyhdr}

\newtheorem{thm}{Theorem}[section]
\newtheorem{exm}[thm]{Example}

\newtheorem{lem}[thm]{Lemma}
\newtheorem{definition}[thm]{Definition}

\newtheorem{cor}[thm]{Corollary}

\newcommand{\Cov}{\texttt{Cov}}

\newcommand{\abs}[1]{\left\vert#1\right\vert}
\newcommand{\lfix}[1]{\lfloor#1\rfloor}

\newcommand{\seq}[1]{\left<#1\right>}
\newcommand{\set}[1]{\left\{#1\right\}}
\newcommand{\Sym}{\texttt{Sym}}

\newcommand{\tr}{\texttt{Tr}}

\newcommand{\var}{\texttt{Var}}
\newcommand{\vecc}{\texttt{vec}}

\newcommand{\al}{\alpha}

\newcommand{\da}{\delta}
\newcommand{\ga}{\gamma}
\newcommand{\Ga}{\Gamma}

\newcommand{\la}{\lambda}

\newcommand{\eps}{\epsilon}
\newcommand{\sig}{\sigma}
\newcommand{\Sig}{\Sigma}
\newcommand{\ifff}{\Leftrightarrow}

\newcommand{\A}{\mathcal{A}}
\newcommand{\B}{\mathcal{B}}
\newcommand{\C}{\mathcal{C}}

\newcommand{\cI}{\mathcal{I}}

\newcommand{\cH}{\mathcal{H}}

\newcommand{\st}{\mathcal{S\!T}}
\newcommand{\T}{\mathcal{T}}

\newcommand{\bft}{\textbf{t}}

\newcommand{\bu}{\textbf{u}}
\newcommand{\bv}{\textbf{v}}

\newcommand{\bx}{\textbf{x}}

\newcommand{\by}{\textbf{y}}

\def\R{\mathbb R}

\newcommand{\mnrt}{$m$th order $n$-dimensional real tensor }
\newcommand{\mnrts}{$m$th order $n$-dimensional real tensors }
\newcommand{\mnst}{$m$th order $n$-dimensional symmetric tensor }
\newcommand{\mnsts}{$m$th order $n$-dimensional symmetric tensors }

\newcommand{\beq}{\begin{equation}}
\newcommand{\eeq}{\end{equation}}
\newcommand{\bey}{\begin{eqnarray}}
\newcommand{\eey}{\end{eqnarray}}
\newcommand{\beyy}{\begin{eqnarray*}}
\newcommand{\eeyy}{\end{eqnarray*}}

\title{High order tensor moments of random vectors}
\author{Yan Feng,  Shan Song, Changqing Xu\thanks{Corresponding author.
Email: cqxurichard@mail.usts.edu.cn}
\footnote{School of Mathematics, Suzhou University of Science and Technology, Suzhou, China.
}}

  \makeatletter
      \def\@setcopyright{}
      \def\serieslogo@{}
      \makeatother
 \date{\today}
\begin{document}
\maketitle

\begin{abstract}
A random vector $\bx\in \R^n$ is a vector whose coordinates are all random variables. A random vector is called a Gaussian vector if it follows Gaussian distribution.
These terminology can also be extended to a random (Gaussian) matrix and  random (Gaussian) tensor.  The classical form of an $k$-order moment  (for any positive 
integer $k$) of a random vector $\bx\in \R^n$ is usually expressed in a matrix form of size $n\times n^{k-1}$ generated from the $k$th derivative of the characteristic 
function or the moment generating function of $\bx$ , and the expression of an $k$-order moment is very complicate even for a standard normal distributed vector. With 
the tensor form, we can simplify all the expressions related to high order moments. The main purpose of this paper is to introduce the high order moments of a random 
vector in tensor forms and the high order moments of a standard normal distributed vector.  Finally we present an expression of high order moments of  a random vector 
that follows a Gaussian distribution.       
\end{abstract}

\noindent \textbf{keywords:} \  Tensor;  random matrix; high order moment; Gaussian distribution; . 


\section{Introduction}
\setcounter{equation}{0}

Higher order moments are important in statistics. The concept of covariance matrix gives rise to that of co-skewness and co-kurtosis when it is extended to the higher order
 moments, say, the third and fourth moments like skewness and kurtosis. This follows from the generalization of the concept of mean and variance to moments and central 
 moments. Higher-order moments of a normal distribution can be used to derive the recursive relationship of Hermite polynomials\cite{Tracy1993}. They are also widely used 
in the insurance industry\cite{ZT2007},color transmission\cite{KR2005}, fault diagnosis\cite{ZHW2009}, large reflector antenna simulation\cite{WLD1988} and other aspects 
also play an important role. The higher order moments are also useful in risk management.  An example would be when the fund performance of four different fund managers 
are analyzed separately and they are then combined together so that in the end only 2 sets of results are compared.  In both cases the moments i.e. the mean, standard deviation, 
skewness and kurtosis for each manager remains the same.\par
\indent  The covariance, i.e., the second order centralized moment of a random variable, determines the holistic divergence from its location (centriod) in one dimension.
The shape and the other features of the distributions of a multidimensional random vector are not so obvious and hard to describe and illustrated by traditional approach.
Note that the variance of $u$ is $D[u]=E[(u-E(u))^2] = E[u^2]-(E[u])^2 =m_2 -m_1^2$.  In statistical analyses, the fundamental tasks include the characterization of the 
location and variability of the distribution of a data set or a population. Further characterization of the data includes the skewness and kurtosis, which involves the computation 
of the third and fourth order moments respectively. The skewness is a measurement of symmetry (or lack of symmetry) of the data distribution.  A data set is said to be
symmetrically distributed if it looks the same to the left and right of the center point.  Kurtosis is a measure of whether the data are heavy-tailed or light-tailed relative to a normal
distribution, that is, data sets with high kurtosis tend to have heavy tails, or outliers, and these with low kurtosis tend to have light tails, or lack of outliers. \par
\indent There are many ways to express the moments, one of the commonly used approach is to use derivatives to the characteristic function or the moment generating
function\cite{Tracy1993, Bjorn1988}. For the standardized 2-dimensional normal distribution, Kendall and Stuart (1963) gave the recurrence relation of the second order 
moment\cite{Kendall1963}. Johnson (2000) and others have given analytical formulas for the same problem. Holmquist (1988) proved the general form of higher-order
moments\cite{JKB2000} and extended the result to include the derivation of normal distribution quadratic higher moments\cite{Bjorn1996}. The problem of moments and
cumulants of normal random matrices is considered by  Ghazal and Neudecker\cite{GN2000}. Using the Kronecker product, a simple formula for the special case of the 
second and fourth moments of the random matrix is derived\cite{GN2000}.\par
\indent In this paper, we mainly introduce the higher-order tensor moments, present some tensor espressions of the higher order moments, and investigate their properties.\par

\indent Recall that a tensor $\A$ is a multi-way array which can be regarded as a hypermatrix. An $m$-order tensor $\A$ can be of size $I_1\times I_2\times \ldots \times I_m$. 
$\A$ is called a \mnrt if $n:= I_1=I_2=\ldots=I_m$.  The set of all \mnrts is denoted as $\T_{m,n}$. For any positive integers $m,n>1$, we usually denote 
$[n]:=\set{1,2,\ldots,n}, [n]_0:=\set{0,1,2,\ldots,n}$ ,  and
\[  S(m,n) = \set{(i_1,i_2,\ldots, i_m): i_k\in [n] ,\forall k\in [m] } \]
and
\[  S(k; m,n) = \set{\sig:=(i_1,i_2,\ldots, i_m)\in S(m,n) : i_1+i_2+\ldots+i_m=m+k } \]
where $k\in [N]_0$ with $N=m(n-1)$.  For any $\tau\in S(m,n)$, it is easy to see that $\tau\in S(0;m,n)$ if and only if  $\tau=(1,1,\ldots, 1)$, the smallest element in set $S(m,n)$ according to the lexical order,
and $\tau\in S(N; m,n)$ if and only if  $\tau=(n,n,\ldots, n)$, the largest element in $S(m,n)$.  \\
\indent An \mnrt  $\A$ with size $n\times n\times\ldots\times n$ is an $m$-array whose entries are indexed by indices $(i_1,i_2,\ldots, i_m)\in S(m, n)$. $\A$’s element $A_{i_1i_2\ldots i_m}$ is also denoted
by $A_{\sig}$ where $\sig=(i_1,i_2,\ldots, i_m)$.  We denote the set of all \mnrts by $T_{m;n}$.  A tensor $\A=(A_{\sig})\in T_{m;n}$ is called a \emph{symmetric tensor} if 
each entry $A_{i_1,i_2,...,i_m}$ is invariant under any permutation of its indices, that is,
\[ A_{\sig}=A_{\tau(\sig)}  \forall  \tau\in \Sym_m, \forall \sig\in S (m,n). \]
where $\Sym_m$ is the set of all permutations on $[m]$.  We denote the set of all \mnsts by $\st_{m;n}$.\par  
\indent In the next section, we will introduce some notations related to the multiplications of tensors, which will be used to characterize higher order moments. Also 
we will define the tensor form of high order derivatives (HOD) of a multivariate function. Some interesting results of 4-order tensors will also be addressed in order to 
prepare for the description of the covariance tensor of a random matrix. \par

\vskip 10pt

\section{The multiplications of tensors and the 4-order tensors}  
\setcounter{equation}{0}

An \mnrt $\A \in \T_{m;n}$ can be associated with an $m$-order $n$-variate homogeneous polynomial in form
\[
f_{\A}(\bx)=\A\bx^m:=\sum_{i_1,i_2,...,i_m}A_{i_1,i_2,...,i_m}x_{i_1}x_{i_2}\ldots x_{i_m}
\]
A symmetric tensor $\A\in \st_{m;n}$ is called \emph{positive semidefinite} or simply \emph{PSD}  if $f_{\A}(\bx) \geq 0$ for all $\bx\in \R^n$ and is called 
\emph{positive definite} (\emph{PD})  if  $f_{\A}(\bx) > 0$ for all nonzero $\bx\in \R^n$.  Let $\A,\B$ be any tensors of order $p$ and $q$ respectively.  Now we 
denote $[p+q]:=\set{1,2,\ldots, p+q}$ and let $[p+q]=S\cup T$ be a proper partition of set $[p+q]$ where the carnalities of  $S$ and $T$ are respectively 
$p$ and $q$.  For convenience, we write $S=\set{s_1,s_2,\ldots,s_p}$ and $T=\set{t_1, t_2,\ldots,t_q}$, both in increasing order.  Then we denote 
$\C:=\A\times_{T}\B$ for the \emph{outer-product} of $\A$ and $\B$, defined by
\beq\label{eq:outprod}
C_{i_1\ldots i_p i_{p+1}\ldots i_{p+q}} = A_{i_{S}}B_{i_{T}}
\eeq
where $i_{S}:= (i_{s_1},i_{s_2},\ldots,i_{s_p}), i_{T}:=(i_{t_1},i_{t_2},\ldots,i_{t_q})$.  $\C$ is called the outer-product of $\A$ with $\B$ along 
mode-$T$, which is a tensor of order $p+q$. Note that the out-product of $m$ (column) vectors produces a tensor of order $m$. We denote 
$\bx^m:= \overbrace{\bx\times \bx\times \ldots\times \bx}^{m}$ for any $\bx\in \R^n$ for our convenience. Thus $\bx^m$ is a rank-1 \mnst. In the following example, 
we consider the outer-product of two $n\times n$ real matrices.  \par   
\begin{exm}\label{exm2-1}
Let $A,B\in \R^{n\times n}$.  There are six different out-products for $(A, B)$, each product $A\times_{\theta} B$ is a 4-order $n$-dimensional tensor where $\theta$ is any 
2-set of $[4]$, i.e., $\theta\in \set{\set{1,2},\set{1,3},\set{1,4},\set{2,3},\set{2,4},\set{3,4}}$. Furthermore, there are three different products when $B=A$, i.e.,
\[ A\times_{(1,2)} A, \quad  A\times_{(1,3)} A, \quad  A\times_{(1,4)} A  \]
since $ A\times_{\theta} A = A\times_{\theta^c} A$ for any 2-set $\theta\subset [4]$.\par
\indent  Note that generally these three tensors are different. For example, let $A=I_n$, the identity matrix. Then we have
\[ (I_n\times_{(1,2)} I_n)_{i_1i_2i_3i_4} = \da_{i_3i_4}\da_{i_1i_2} \] 
and 
\[ (I_n\times_{(1,3)} I_n)_{i_1i_2i_3i_4} = \da_{i_2i_4}\da_{i_1i_3} \] 
for any index $(i_1,i_2,i_3,i_4)\in S(4,n)$, where $\da_{ij}$ is the Kronecker constant ($\da_{ij}\in \set{0,1}$ and $\da_{ij}=1\ifff j=i$).       
\end{exm}
\indent Sometimes we need to reduce or preserve the order of tensors by multiplication. For this purpose, we introduce the contractive multiplications of tensors, which may be regarded 
as the extension of Einstein multiplications of tensors.  Let $\A\in \T_{p;n}, \B\in \T_{q;n}$ and let $S\subset [p],T\subset [q]$ with $r=\abs{S}=\abs{T}$ ($1\leq r\leq \min(p,q)$). Here 
$\abs{S}$ denotes the cardinality of  a set $S$ 
Denote $m=p+q-2r$. The \emph{Einstein product} of $\A$ with $\B$ along mode-$(S,T)$ as an $m$-order tensor $\C:=\A\times_{(S,T)}\B$ which is defined by 
\[ C_{\theta} = \sum\limits_{\eta_S} A_{\theta_S} B_{\tau_T} \]
where $\theta_S\in S(p,n), \tau_T\in S(q,n)$.  For example, if $\A,\B\in \T_{4;n}$, and $S=\set{3,4}, T=\set{1,2}$. Then we have $\C:=\A\times_{(S,T)}\B\in \T_{4;n}$ whose entries are 
\[  C_{i_1i_2i_3i_4} = \sum\limits_{j_1, j_2} A_{i_1 i_2 j_1 j_2}B_{j_1 j_2 i_3i_4} \]
where the summation is over all $j_1,j_2\in [n]$.  Moreover, if  $\A\in \T_{m;n}, B\in \R^{n\times p}$.  Then an $k$-mode multiplication of $\A$ by $B$ from the right side, denoted 
$\A\times_k B$, is defined by
\beq\label{eq: kmodeprod}
(\A\times_k B)_{i_1\ldots i_{k-1} i_k i_{k+1}\ldots i_n} = \sum_{j =1}^n  A_{i_1,\ldots i_{k-1} j i_{k+1}\ldots i_n}B_{j i_k}
\eeq
Sometimes we briefly denote it by $\A B$ when $k=n$.  Similarly,  the $k$-mode multiplication of $\A$ by $B$ from the left,  denoted $B\times_k \A$, is defined by
\beq\label{eq:MatrxTensor}
(B\times_{k}\A)_{i_1\ldots i_{k-1} i_k i_{k+1}\ldots i_m}=\sum\limits_{j} A_{i_1\ldots i_{k-1} j i_{k+1}\ldots i_m} B_{i_k j}
\eeq
We denote $[B]\A=B\times_1\times_2\times \ldots \times_k \A$ when $\A\in \T_{m;n}, B\in \R^{n\times n}$.\par   

\indent   (\ref{eq: kmodeprod}) conforms to matrix multiplication.  For example, we have 
\beq\label{eq:MwithM} 
A\times_{1}[B]=A^{\top}B, A\times_{2}[B] = AB,  [B]\times_1 A = BA, [B]\times_2 A =BA^{\top} 
\eeq   

The contractive product of an \mnst $\A$ with an $n$-dimensional vector $\bx$ in all modes yields an $m$-degree $n$-variate homogeneous polynomial 
$f(\bx):=\A\bx^m$,  and $\by:=\A\bx^{m-1}$, which is defined as a vector $\by=(y_1,y_2,\ldots,y_n)^{\top}$ with  
\[ y_i = \sum\limits_{i_2,i_3,\ldots,i_m} A_{ii_2i_3\ldots i_m}x_{i_2}x_{i_3}\ldots x_{i_m}, \qquad   i=1,2,\ldots,n\]  
can be used to define the eigenvalues and eigenvectors of a tensor.  
\indent Now we consider the linear space $\T_{4;n}$, the set of all 4-order $n$-dimensional real tensors. Let $\A,\B\in \T_{4;n}$. The product $\C=\A\times\B\in \T_{4;n}$ is 
defined as 
\beq\label{eq;4tprod} 
C_{i_1i_2i_3i_4} =\sum\limits_{j_1, j_2} A_{i_1i_2j_1j_2}B_{j_1j_2i_3i_4} 
\eeq
for any $(i_1,i_2,i_3,i_4)\in S(4,n)$. In this case, we may define the identity tensor $\cI=(\eps_{i_1i_2i_3i_4})\in \T_{4,n}$ as $\eps_{i_1i_2i_3i_4}=\da_{i_1i_3}\da_{i_2i_4}$
where $\da_{ij}$ is the Kronecker constant, i.e., $\da_{ii}=1$ and $\da_{ij}=0$ for all distinct $i,j$. It is easy to see that $\cI = I_n\times_{(1,3)}I_n$. We can also 
show that 
\begin{lem}\label{le: identitytensor01}
For any tensor $\A\in \T_{4,n}$, we have 
\beq\label{eq: identitytensorprop01}
\A\times \cI = \cI\times \A = \A 
\eeq
\end{lem}

\begin{proof}
We show the equality $\A\times \cI =\A$. For any given index $(i_1,i_2,i_3,i_4)\in S(4,n)$, we have 
\beyy  
    (\A\times \cI)_{i_1i_2i_3i_4} &=& \sum\limits_{j_1,j_2} A_{i_1i_2j_1j_2}\cI_{j_1j_2i_3i_4} \\
                                                 &=& \sum\limits_{j_1,j_2} A_{i_1i_2j_1j_2}\da_{j_1i_3}\da_{j_2i_4}\\
                                                 &=& A_{i_1i_2i_3i_4}
\eeyy
Thus we have  $\A\times \cI =\A$. Similarly we can also prove the equality $\cI\times \A =\A$.                                                                    
\end{proof}

\indent We are now ready to define the tensor form of high order derivatives (HOD) of a multivariate function. Let $f(\bx)=f(x_1,x_2,\ldots,x_n)$ be the function
defined on $\R^n$ which maps $\R^n$ to $\R$. Then the first derivative of $f$ with respect to $\bx$, also called the gradient of $f$, denoted by
$\frac{df}{d\bx}$, is defined as $\frac{df}{d\bx}:=(\frac{df}{dx_1},\frac{df}{dx_2},\ldots,\frac{df}{dx_n})^{\top}\in \R^n$. The second derivative
of $f$ is defined accordingly by $\frac{d^2 f}{d\bx^2} =\frac{d}{d\bx}(\frac{df}{d\bx})$, which yields the definition
\beq\label{eq:2der}
\frac{d^2 f}{d\bx^2} =(h_{ij}): \  h_{ij} = \frac{d^2 f(\bx)}{dx_i dx_j}
\eeq
$H=(h_{ij})\in \R^{n\times n}$ is called the \emph{Hessian matrix} of $f(\bx)$. Note that sometimes the Hessian of $f$ refers to the determinant of
the Hessian matrix. Here we only concern the Hessian matrix. The higher order derivatives of $f(\bx)$, i.e., $\frac{d^k f}{d\bx^k}$ for $k>2$, is bit
of more complicate traditionally since all $\frac{d^k f}{d\bx^k}$ are defined in matrix form which is achieved by recursive vectorization of the matrix
$\frac{d^(k-1) f}{d\bx^(k-1)}$ before the next derivative, i.e.,
\beq\label{eq:k+1der}
H^{(k+1)}:= \frac{d^{(k+1)} f}{d\bx^{(k+1)}} =  \frac{d}{d\bx} (\vecc(\frac{d^k f}{d\bx^k}) )
\eeq
Thus $\frac{d^k f}{d\bx^k}\in \R^{n\times n^{k-1}}$ for all $k\geq 2$. Note that $H^{(1)}\in \R^n$ and $H=H^{(2)}\in \R^{n\times n}$ is exactly the Hessian of $f$.
This conventional definition of HOD of $f$ ambiguous the meaning of each element when $k>2$. A more natural definition is the following:
\beq\label{eq:kdertensor}
\cH^{(k)}:=(H_{i_1i_2\ldots i_k}), \quad  H_{i_1i_2\ldots i_k} = \frac{d^k f}{dx_{i_1}dx_{i_2}\ldots dx_{i_k}}
\eeq
Thus we call $\cH^{(k)}$ the $k$-order Hessian tensor of $f(\bx)$. Note that $\cH^{(k)}$ is symmetric due to the commutavity of derivatives of $f$.  Hence we 
have $\cH^{(k)}\in \st_{k;n}$. \par 
\indent  Now we consider the derivatives of a matrix variable $Y\in \R^{m\times n}$ with respect to another matrix variable $X\in \R^{m\times n}$
(taking each entry of $Y$ as a function of  the elements of $X$).  Conventionally this is defined as a matrix $H=(H_{ij})\in \R^{mn\times mn}$ with 
$H_{ij}=\frac{\partial y_{i _1 j_1}}{\partial x_{i _2 j_2}}$ where 
\beq\label{eq: divbym}   
i = (j_1-1)m +i _1, \quad   j = (j_2-1)m +i _2 
\eeq  
where $0\leq i_1,i_2<m$ and $1\leq j_1,j_2\leq n$. Note that 
\[ \frac{\partial Y}{\partial X}= \frac{\partial \vecc(Y)}{\partial \vecc(X)^{\top}}  \]
It follows that $\frac{\partial y}{\partial x}=(h_{ij})$ with $h_{ij}=\frac{\partial y_i}{\partial x_j}$ when $x,y$ are both vectors.\par  
\indent The two equations in (\ref{eq: divbym}) are obtained from the division theorem with remainder properties. 
Note that when $i_1=0$ ($i_2=0$), we replace $j_1$($j_2$) by $j_1-1$($j_2-1$) and let $i_1=n $ ($i_2=n$).  The definition provokes some conveniences 
especially when coping with the $k$th derivative of $Y$ w.r.t. $X$, as seen from the above.  Now we introduce the tensor form of the derivatives.  \par 
\begin{definition}\label{def: derivtensor}   
Let $X, Y\in \R^{m\times n}$ where each entry $y_{ij}$ is regarded as a function of  $mn$ variables $\set{x_{ij}}$.  We define the derivative $\frac{\partial Y}{\partial X}$
as the 4-order tensor $\A = (A_{i_1i _2j_1j_2})$ with  
\beq\label{eq: def1dertensor}  
A_{i_1i _2j_1j_2} = \frac{\partial y_{i _1 j_1}}{\partial x_{i _2 j_2}}
\eeq 
which is of size $m\times m\times n\times n$. Now recursively we define the $k$th derivative as 
\[ \frac{\partial^{k} Y}{\partial X^{k}} = \frac{\partial}{\partial X} \frac{\partial^{k-1} Y}{\partial X^{k-1}} \]
If we denote $\A^{(k)}:= \frac{\partial^{k} Y}{\partial X^{k}}$, then $\A^{(k)}$ is an $2(k+1)$-order tensor, with 
\beq\label{eq: kdertensor}  
A_{i_1i _2\ldots i _k i_{k+1} j_1j_2\ldots j _k j_{k+1}} = \frac{\partial^{k} y_{i _1 j_1}}{\partial x_{i _2 j_2}\ldots \partial x_{i _k j_k}\partial x_{i _{k+1} j_{k+1}}}
\eeq  
When both $X,Y$ are reduced to vectors, $\A^{(k)}$ reduces to an  $(k+1)$-order tensor. \par 
\end{definition}
\indent  In the next section, we will use the derivative tensors to present the higher order moments of random vectors and random matrices. \par 
   
\vskip 5pt

\section{ High order tensor moments}
\setcounter{equation}{0}

High order moments can be expressed in the form of tensors which can simplify their expressions. The covariance of a random variable $x$ is 
the second central moment of $x$,i.e., $\var(x)=E[(x-E[x])^2]$, and the covariance matrix of a random vector $\bx\in \R^n$, defined as 
\[ \var(\bx)=E[(\bx-E[\bx])(\bx-E[\bx])'] = E[\bx\bx^{\top}]-E[\bx] E[\bx]^{\top}=m_2 -\mu^2 \]  
can be regarded as a function of $m_1$ and $m_2$.  Furthermore, the covariance matrix of a random matrix $X\in \R^{m\times n}$ is conventionally defined as 
\beq\label{eq:defcov4matrx}
\Cov(X)=E[(\bx-\mu)\times (\bx-\mu)] 
\eeq
(note that $\bx\times \by=\bx\by^{\top}$ for any column vectors $\bx,\by$) where $\bx=\vecc(X), \mu=\vecc(E[X])=E[\vecc(X)]\in \R^{mn}$. Thus $\Cov(X)$ is 
a PSD $mn\times mn$ matrix. However, the definition (\ref{eq:defcov4matrx}) ruins the structure of $X$ and thus makes the interpretation of each entry of $\Cov(X)$ 
vague.  A more natural expression for the covariance of a random matrix $X\in \R^{m\times n}$ should be a tensor of order 4, as in the following:
\beq\label{eq:defcovtensor4matrx}
\C:=\Cov(X)=E[(X-E[X])\times_{(1,3)} (X-E[X])] 
\eeq
Here we use outer-product $A\times_{(1,3)} A$, which is the same to $A\times_{(2,4)} A$, to make the size of tensor $\C$ as $m\times m\times n\times n$. 
Specifically, the 4-order tensor $\C=(C_{i_1i_2j_1j_2}) \in \R^{m\times m\times n\times n}$ is defined by 
\beq\label{eq: covtele4matr}
C_{i_1i_2j_1j_2} = \Cov(x_{i_1j_1}, x_{i_2j_2}) = E[(x_{i_1j_1}-\mu_{i_1j_1})(x_{i_2j_2}-\mu_{i_2j_2})]  
\eeq 
where $\mu=(\mu_{ij})=E[X]\in \R^{m\times n}$.\par
\indent An important tool for deriving moments is the characteristic function (CF).  Let $u\in \R$ be a random variable.  We denote by $\varphi_u(t)$ 
the CF of $u$, which is defined by 
\beq\label{cf01}
\varphi_u(t)=E[e^{\imath tu}]
\eeq
where $\imath$ is the imaginary unit. Let the CF $\varphi_\bx(\bft)$ be $k$ times differentiable. Then the $k$-moment of  $\bx\in \R^{n}$ equals
\beq\label{momentbyCF}
m_k[\bx]=\frac{1}{\imath^k}\frac{d^k}{d \bft^k}\varphi_\bx(\bft)\mid_{\bft=0}
\eeq
Similarly the $k$-central moment of a random vector $\bx\in \R^{n}$ is given by
\beq\label{eq:centrM01}
\bar{m}_k[\bx]=m_k[\bx-E[\bx]]=\frac{1}{\imath^k}\frac{d^k}{d\bft^k}\varphi_{\bx-E[\bx]}(\bft)\mid_{\bft=0},\quad \bft\in{R^p}
\eeq
From (\ref{momentbyCF}) we have $m_1\in \R^n$ and $m_2$ is an $n\times n$ matrix. The definition of $m_k$ for $k>2$ is constraint by that of the high order derivative
of a multivariate function.  \par
\indent  A matrix is called a random matrix if  each of its entries is a random variable.  A random matrix $X\in \R^{p\times q}$  can be regarded as a consequence of
matricization of a random vector $\vecc(X)\in \R^{pq}$.  We suppose the characteristic function $\varphi_X (T)$ of  $X$ be $k$ times differentiable with 
$T\in \R^{p\times q}$ being any arbitrary matrix. Then the $k$-moment of $X$ is defined by 

\beq\label{eq:kMbyCF4M}
m_k[X]=\frac{1}{\imath^k}\frac{d^k}{dT^k}\varphi_{X}(T)\mid_{T=0}
\eeq
Similarly the $k$-central moment $\bar{m}_k[X]$ of $X$ is defined by 
\beq\label{eq:kcentrMbyCF4M}
\bar{m}_k[X]=\frac{1}{\imath^k}\frac{d^k}{dT^k}\varphi_{X-E[X]}(T)\mid_{T=0} 
\eeq
Sometimes we denote by $m_k$($\bar{m}_k$) instead of $m_k[X]$ ($\bar{m}_k[X]$) if there is no risk of confusion.  Note that $m_k$($\bar{m}_k$)  is a tensor of order $2k$ whose size is  
\[ \overbrace{m\times m\times \ldots \times m}^{k}\times \overbrace{n\times n\times \ldots \times n}^{k} \]
when $X\in \R^{p\times q}$.  We have 
\begin{thm}\label{th:kmomenttensor}
Let $X=(x_{ij})\in \R^{p\times q}$ be a random matrix and let $M=E[X]$ be the mean matrix of  $X$.  Then its $k$-moment tensor and the $k$-central moment tensor are respectively 
\beq\label{eq:kMtensor}
m_k[X]=E[X^{\times k}] 
\eeq
and 
\beq\label{eq:kcentrMtensor}
\bar{m}_k[X]=E[(X-M)^{\times k}] 
\eeq
where $X^{\times k}:=\overbrace{X\times X\times \ldots \times X}^{k}$ is the $k$th power of $X$ in the sense of outer-product. 
\end{thm} 

\begin{proof}
 We denote by $\phi=\phi_{X}(T)$ for simplicity.  By the tensor form of the higher order derivatives  defined in (\ref{eq: kdertensor}),  we have 
 \beq\label{eq:kder4CF}
\frac{d^k \phi}{dT^k}=\frac{1}{\imath^k} E[\exp{\imath \seq{T, X}}X^{\times k}] 
\eeq
Thus by  (\ref{eq:kMbyCF4M}), we get (\ref{eq:kMtensor}).  Similarly we can prove  (\ref{eq:kcentrMtensor}).
\end{proof}

\indent  For any given positive integers $n, s$ where $1\leq s\leq n$, we denote 
\[ \pi_s(n) := \set{\theta_s:=(i_1,i_2,\ldots,i_s): 1\leq i_1<i_2<\ldots < i_s\leq n } \]
and $\pi_0 :=\emptyset$ (the empty set ). Now let $\A, \B$ be tensors of order $p$ and $q$ respectively, and let $p+q=n$.  
The outer-product $\A\times_{\theta_s} \B$ is a tensor of order $n$ as defined by (\ref{eq:outprod}).  For $s=0$, $\pi_s(n) =\emptyset$, and we denote  
 $\A\times_{\emptyset} \B =\A$. On the other hand, we denote $\A\times_{[n]} \B=\B$ if $s=n=q$. \par 
\indent From Theorem \ref{th:kmomenttensor}, we have 
\begin{cor}\label{cor:kmoment4vec}
Let $\bx = (x_{j})\in \R^{n}$ be a random vector with mean vector $\mu = E[\bx]$.  Then its $k$-moment (central moment) is the $k$-order $n$-dimensional 
tensor $m_k[\bx] = E[\bx^{\times k}]$ ($\bar{m}_k[X] = E[(\bx-\mu)^{\times k}]$).  Furthermore, we have 
 \beq\label{eq:kcentrm2m}
\bar{m}_k = \sum\limits_{s=0}^k (-1)^s \sum\limits_{\theta_s\in \pi_s} m_{k-s}\times_{\theta_s} \mu^s 
\eeq
\end{cor}

\begin{proof}
The first part of the corollary is immediate from Theorem \ref{th:kmomenttensor}, and (\ref{eq:kcentrm2m}) can be obtained by 
\beyy
\bar{m}_k  &=& E[(\bx-\mu)^{\times k}] \\
                    &=& E[\sum\limits_{s=0}^k (-1)^s \sum\limits_{\theta_s\in \pi_s(k)} {\bx}^{k-s}\times_{\theta_s}\mu^s ]\\
                    &=& \sum\limits_{s=0}^k (-1)^s \sum\limits_{\theta_s\in \pi_s(k)} E[{\bx}^{k-s}]\times_{\theta_s} \mu^s\\
                   &=& \sum\limits_{s=0}^k (-1)^s \sum\limits_{\theta_s\in \pi_s(k)} m_{k-s}\times_{\theta_s} \mu^s 
\eeyy 
\end{proof}
\indent By Corollary \ref{cor:kmoment4vec}, we have $\bar{m}_2=m_2 -\mu^2 \in \R^{n\times n}$ which is the covariance of  $\bx$, and
\beq\label{eq: centrm3}
\bar{m}_3 = m_3 - \sum\limits_{k=1}^3 m_2\times_{k} \mu +2\mu^3
\eeq


\indent  We note that the entries of the $k$-moment of  a random vector $\bx\in \R^n$, by Corollary \ref{cor:kmoment4vec}, is 
\beq\label{eq: entryofMk} 
(m_k)_{i_1i_2\ldots i_k} = E[x_{i_1} x_{i_2}\ldots x_{i_k}]  
\eeq
which conforms to the traditional definition when $m_k$ is flattened to a matrix.      

\vskip 10pt

\section{Higher order moments of multivariate Gaussian distribution}
\setcounter{equation}{0}

Gaussian distribution is the most basic and important distribution in statistics.  The density function of a Gaussian vector $\bx\in \R^n$ with mean vector $\mu\in \R^n$ and 
covariance $\Sig\in \R^{n\times n}$ (usually assumed to be nonsingular), is 
\beq\label{pdf4gvec}
f_{\bx}(\bft)=(2\pi)^{-n/2}\det(\Sig)^{-1/2}\exp\set{-\frac{1}{2}\tr\set{\Sig^{-1}(\bft-\mu)(\bft-\mu)^{\top}}}
\eeq
where $\bft\in \R^n$ is arbitrary.  A random matrix $U\in \R^{m\times n}$ is called a \emph{Standard Normal matrix} or a $SN$-matrix if  
\beq
\vecc(U)\sim N_{mn}(0,I_{mn})
\eeq
where $I_k$ stands for the identity matrix.  A SN-matrix $U\in \R^{m\times n}$ is denoted by $U\sim N_{m,n}(0, I_m, I_n)$, meaning 
that all the columns $u_j$($1\leq j\leq n$) of $U$ are i.i.d. with $u_j\sim N_m(0, I_m)$ and all rows $w_i$ ($1\leq i\leq m$) of $U$ are i.i.d. with $w_i\sim N_n (0, I_n)$. 
A random matrix $X\in \R^{m\times n}$ is called a \emph{Gaussian} matrix if there exist constant matrices $\mu\in \R^{m\times n}, A\in \R^{m\times p}, B\in \R^{n\times q}$ 
such that  
\beq\label{eq:defGmatrix} 
X \overset{d}{\sim} \mu+AUB^{\top},\qquad U\sim N_{p,q}(0, I_p, I_q)
\eeq 
where $U \overset{d}{\sim} V$ means that two random variables (vectors, matrices) $U,V$ have the same distribution. This is denoted by $X\sim N_{m,n}(\mu, \Sig_1, \Sig_2)$.
In this situation, we call $X$ is a Gaussian matrix with parameters $(\mu, \Sig_1, \Sig_2)$ where $\Sig_1:=AA^{\top}, \Sig_2:= BB^{\top}$.\par 

\indent  Let $\mu\in \R^{m\times n}, \Sig_1\in \R^{m\times m}, \Sig_2\in \R^{n\times n}$ with $\Sig_k$ being PSD.  It is shown\cite{KR2005} that 
\begin{lem}\label{le:eqivcond4Gmatrx}
Let $X\in \R^{m\times n}$ be a random matrix. Then $X\sim N_{m,n}(\mu, \Sig_1, \Sig_2)$ if and only if it satisfies 
\begin{description}
\item[(1). ]  $E[X]=\mu\in \R^{m\times n}$.
\item[(2). ]  $X_{\cdot{}j}\sim N_m(\mu_{\cdot{}j},\sig^{(2)}_{jj}\Sig_1)$ for all $j\in [n]$.   
\item[(3). ]  $X_{i\cdot{}}\sim N_n(\mu_{i\cdot{}},\sig^{(1)}_{ii}\Sig_2)$ for all $i\in [m]$. 
\end{description}
where $A_{\cdot{}j}$($A_{i\cdot{}}$) denotes the $j$th column ($i$th row) of  matrix $A$, and $\Sig_k=(\sig_{ij}^{(k)})$ for $k=1,2$.  
\end{lem}
\indent  An immediate corollary from Lemma \ref{le:eqivcond4Gmatrx} is 
\begin{cor}\label{co: Gmatrx2Gvec}
Let $\mu\in \R^{m\times n}, \Sig_1\in \R^{m\times m}, \Sig_2\in \R^{n\times n}$ with $\Sig_k$ ($k=1,2$) being PSD. Then  $X\sim N_{m,n}(\mu, \Sig_1, \Sig_2)$ 
implies   
\beq\label{eq: Gmatrx2vec}
\vecc(X) \sim N_{mn} (\vecc(\mu), \Sig_2\otimes \Sig_1)
\eeq
Furthermore, if $\Sig_k$($k=1,2$) both are positive definite, then the density function of $X$ is
\beq\label{eq: density4Gmatrix}
f_X (T)=(2\pi)^{-mn/2}\det(\Sig_1)^{-n/2}\det(\Sig_2)^{-m/2}\exp\set{\psi(T)}
\eeq
where $T\in R^{m\times n}$ is arbitrary and $\psi(T):=-\frac{1}{2}\tr[\Sig_1^{-1}(T-\mu)\Sig_2^{-1}(T-\mu)^{\top}]$. 
\end{cor}

\indent  Given an even integer $k=2m, m\geq 1$.  A partition $\ga:=\set{\ga_1,\ga_2,\ldots,\ga_m}$ of set $[k]$ is called a \emph{2-partition} if  
$[k]=\ga_1\cup \ga_2\cup \ldots \cup \ga_m$ with $\abs{\ga_j}=2$ for each $j\in [m]$. Denote by $\Ga_2[k]$ the set of all 2-partitions of 
$[k]$  and let $a_m:=\abs{\Ga_2[k]}$ denote the cardinality of $\Ga_2[m]$. Then 
\beq\label{eq: 2partinumb}
a_m = (2m-1)!! 
\eeq 
where we define $a_m=1$ when $m\le 1$.  There are many methods (e.g. the graph theory) to prove  (\ref{eq: 2partinumb}). Since a 2-partition of set $[2m]$ corresponds to a 1-factor 
of a complete graph $K_{2m}$, $a_m$ is exactly the number of the 1-factors of $K_{2m}$, which satisfies the recurrence $a_m =(2m-1)a_{m-1}$, by which (\ref{eq: 2partinumb}) follows. \par 
\indent  Now we extend the 2-partitions of a set $[k]$ for any positive integer $k$.  For any integer $s$ with $k\ge 2s\ge 0$, we let $W$ be a subset of $[k]$ with $\abs{W}=k-2s$, and 
$\ga:=\set{\ga_1,\ldots,\ga_s}$ be a 2-partition of  $W^{c}:=[k]\backslash W$ ($W=\emptyset$ if $s=\lfix{k/2}$).  Define 
\[ \bar{\ga}:= \set{\ga_1,\ga_2,\ldots,\ga_s,\ga_{s+1}} \] 
with $\ga_{s+1}=W$\footnote{$\ga_{s+1}$ may be an empty set.}.  Then $\bar{\ga}$ is a partition of  $[k]$.  We call $\bar{\ga}$ a \emph{$[s,2]$-partition} of $[k]$,  and denote 
$\Pi(s, k)$ the set of all $[s,2]$-partitions of $[k]$.  \par 
\indent  A 2-partition $\ga$ uniquely determines the pattern of the $m$th power of a matrix $A$ in terms of the outer product. 
For example, when $m=2$($k=4$), we have three 2-partitions of set $[4]:=\set{1,2,3,4}$, that is, 
\[  \set{1,2}\cup \set{3,4}, \quad \set{1,3}\cup \set{2,4}, \quad  \set{1,4}\cup \set{2,3}. \]
Thus we have three different patterns of  $I_n\times I_n\times I_n$, i.e.,   
\[ I_n\times_{(1,2)} I_n, \quad  I_n\times_{(1,3)} I_n, \quad  I_n\times_{(1,4)} I_n. \]
For any matrices $A_1,A_2,\ldots,A_m$ and any $\ga\in \Ga_2[m]$, we denote 
\beq\label{eq:gammapower} 
A^{\ga}:=A_1\times_{\ga_2} A_2 \times_{\ga_3} A_3\times\ldots \times_{\ga_m} A_m 
\eeq
Note that (\ref{eq:gammapower}) is weel-defined since the outer-product satisfies the associativity law. Moreover, (\ref{eq:gammapower})  is denoted by 
$I_n^{\ga}$ when $A_1=A_2=\ldots =A_m=I_n$. Obviously $I_n^{\ga} \in \T_{k;n}$.  For any given index $\sig:=(i_1,i_2,\ldots, i_k)$ ($k=2m$), we have 
\beq\label{eq: eleofidentity}
 (I_n^{\ga})_{\sig} =\da_{i_{\ga_1}}\da_{i_{\ga_2}}\ldots \da_{i_{\ga_m}} 
\eeq
where $\da_{i_{\ga_l}}:=\da_{i_s i_t}$ if $\ga_l=\set{s,t}$. \par 

\indent  The following result gives the expressions for $k$-order moment of a SND vector.

\begin{thm}\label{th: M4sndvec}
Let $\bu\in \R^n$ be a SND random vector and let $m_k$ denotes the $k$-order moment of $\bu$. Then 
\begin{description}
\item[(1).]  $m_k = 0\in \T_{k;n}$ for all odd integer $k=1,3,5, \ldots$. 
\item[(2).]  For all even intgers $k=2m$, we have 
\beq\label{eq: Mk-even} 
m_k = \sum\limits_{\ga\in \Ga_2} \cI_n^{\ga}
\eeq  
\end{description}
\end{thm}

\begin{proof} 
Let $\bu=(u_1,u_2,\ldots,u_n)^{\top}$ where $n>1$.  For any given positive integers $k>1$, we denote 
\[ R[n,k]:=\set{\al=(r_1,r_2,\ldots,r_n): r_1+r_2+\ldots+r_n=k, r_s\in [k]_0, \forall s\in [n] } \] 
For any $\phi=(i_1,i_2,\ldots,i_k)\in S(m,n)$, we call $\phi$ an $\al$-type where $\al=(r_1,r_2,\ldots,r_n) \in R[n,k]$, if for each $s\in [k]$,
\[ r_s = \abs{\set{t: i_t =s, t=1,2,\ldots, k}} \] 
that is equivalent to the condition $x_{i_1}x_{i_2}\ldots x_{i_k}=x_1^{r_1}x_2^{r_2}\ldots x_n^{r_n}$ for any vector $\bx=(x_1,x_2,\ldots,x_n)^{\top}$. 
For any index $\tau:=(i_1,i_2,\ldots, i_k)\in S(k,n)$, suppose $\tau$ is $\al$-type where $\al=(r_1,r_2,\ldots,r_n) \in R[n,k]$. Then   
\beyy\label{eq: kmequivform}
(m_k)_{i_1i_2\ldots i_k}  &=& E[u_{i_1}u_{i_2}\ldots u_{i_k}]\\
                                        &=& E[u_1^{r_1}u_2^{r_2}\ldots u_n^{r_n}]\\
                                        &=& E[u_1^{r_1}] E[u_2^{r_2}] \ldots E[u_n^{r_n}]   
\eeyy 
The last equality follows from the fact that $u_1,u_2,\ldots,u_n$ are independent since $\bu\sim N_n(0, I_n)$. Note that $E[u_j^{r_j}]=1$ if $r_j=0$. \par
\indent Now we prove the first item. Let $k>1$ be any odd integer and $\tau:=(i_1,i_2,\ldots, i_k)\in S(k,n)$ be an $\al$-type. Then there exists $s\in [n]$ 
such that $r_s$ is odd, it follows that $E[u_s^{r_s}]=0$ since $u_s\sim N(0,1)$. By (\ref{eq: kmequivform}) we immediately get 
$(m_k)_{i_1i_2\ldots i_k}=0$. Consequently we have $m_k=0$ for all odd integer $k$. \par 
\indent  To prove the second item, we let $k=2m$ ($m=1,2,\ldots, $) and denote the right hand side of (\ref{eq: Mk-even}) by $\A$.  For any given $\sig\in S(k,n)$, 
let $\sig$ be a $(r_1,r_2,\ldots,r_n)$-type. We want to show that $A_{\sig} = (m_k)_{\sig}=\la_{\sig}$ where $\la_{\sig}$ is defined as 
\beq\label{eq: deflambda} 
\la_{\sig}= \prod\limits_{i=1}^{n}(r_i -1)!! 
\eeq 
For this purpose, we write 
\[ P(\sig) :=\set{j\in [n]: r_j >0 } =\set{j_1,j_2,\ldots,j_T },\] 
and let $\abs{P(\sig)} = T$, i.e., the number of positive $r_i$s, which is related to $\sig$. We call $\sig$ an \emph{essentially $r[P(\sig)]$-type} index. Then   
\beyy
(m_k)_{\sig}  &=& E[x_{i_1}x_{i_2}\ldots x_{i_k} ] \\
                      &=& E[x_1^{r_1}x_2^{r_2}\ldots x_n^{r_n}] \\
                      &=& E[x_{j_1}^{r_{j_1}} x_{j_2}^{r_{j_2}} \ldots x_{j_T}^{r_{j_T}}]                       
\eeyy 
It follows that 
\beq\label{eq: Meleofr-type} 
(m_k)_{\sig} = E[x_{j_1}^{r_{j_1}}] E[x_{j_2}^{r_{j_2}}] \ldots E[x_{j_T}^{r_{j_T}}]
\eeq 
If there is a $t\in [T]$ such that $r_{j_t}$ is an odd number, then $E[x_{j_T}^{r_{j_T}}]=0$ by (1) and thus $(m_k)_{\sig}=0$ by (\ref{eq: Meleofr-type}).  
It follows that each nonzero entry of $m_k$ is associated with a $\sig\in S(k,n) $ of a $(r_1,r_2,\ldots,r_n)$-type where each $r_i$ is even (including 0). 
This fact is coincident with that of $\A$ as we can verify by simple deduction. Furthermore, we have 
\[ 
A_{\sig} = \sum\limits_{\ga} \da_{i_{\ga_1}}\da_{i_{\ga_2}}\ldots \da_{i_{\ga_m}} 
\] 
by (\ref{eq: eleofidentity}). Since $\sig$ is $(r_1,r_2,\ldots,r_n)$-type or essentially $r[P(\sig)]$-type where each $r_j$ is even, there are 
\beq  
(r_{j_1}-1)!! (r_{j_2}-1)!!\ldots (r_{j_T}-1)!!  
\eeq 
 
2-partitions $\ga$ of $\set{i_1,i_2,\ldots,i_k}$ such that  
$\da_{i_{\ga_1}}\da_{i_{\ga_2}}\ldots \da_{i_{\ga_m}}=1$. It turns out that  $A_{\sig} = \la_{\sig}$.  On the other hand, we have 
\beyy
(m_k)_{\sig}  &=& E[x_{j_1}^{r_{j_1}} x_{j_2}^{r_{j_2}} x_{j_T}^{r_{j_T}}] \\
                     &=& E[x_{j_1}^{r_{j_1}}] E[x_{j_2}^{r_{j_2}}] \ldots x_{j_T}^{r_{j_T}} \\
                     &=& (r_{j_1}-1)!! (r_{j_2}-1)!!\ldots (r_{j_T}-1)!! \\
                     &=& \la_{\sig}
 \eeyy
Consequently we have $A_{\sig} = (m_k)_{\sig}$ for all $\sig=(i_1,i_2,\ldots, i_k)\in S(k,n)$ for $k=2m$. The proof is completed. 
 \end{proof}
\indent  It is obvious from the proof of Theorem \ref{th: M4sndvec} that 
\begin{cor}\label{co: 4moment4snd} 
 Let $x\sim N(0,1)$ be a SND random variable.  Then its $2n$-order moment $m_{2n}=(2n-1)!!$.  
\end{cor}

\begin{cor}\label{co: 4moment4snd} 
 Let $\bu\in \R^n$ be a SND random vector.  Then its $4$-order moment $m_4$ is 
\beq\label{eq: m4}
  m_4 =  I_n\times_{\set{1,2}} I_n + I_n\times_{\set{1,3}} I_n + I_n\times_{\set{1,4}} I_n
\eeq
\end{cor}
\indent Corollary \ref{co: 4moment4snd} can be deduced directly by Theorem \ref{th: M4sndvec} and Example 4.1. Here we present an alternative proof to double
check the result from the different aspects. \par

\begin{proof}  For convenience, we denote by $\A$ the right hand side of  (\ref{eq: m4}). Then $\A=(A_{ijkl})\in \T_{4;n}$. We need to show that $M_{ijkl}=A_{ijkl}$ 
for all $\set{i,j,k,l}\in S(4,n)$ where $M_{ijkl}$ is the element of $m_4$ indexed by $(i,j,k,l)$. Denote by $t$ the number of distinct elements in $\set{i,j,k,l}$.  Then 
$t\in [4]$. We need only to consider the following five cases in terms of $t$ based on the symmetry of  $m_4$ and $\A$. \par
\noindent (1)   $t=1$, i.e., $i=j=k=l\in [n]$.  Then $A_{iiii}=3$ by the definition of $\A$.  On the other hand,
\beyy
M_{iiii}  & = & (2\pi)^{-n/2} \int_{\R^n} u_i^4 \exp\set{-\frac{1}{2}\bu^{\top}\bu} d\bu \\
               & =  & (2\pi)^{-1/2} \int_{-\infty}^{+\infty} u_i^4 \exp\set{-\frac{1}{2}u_i^2} du\\
               & =  & 3
\eeyy
Thus we have $M_{iiii}=A_{iiii}$ for all $i\in [n]$. \par
\noindent (2)   $t=2$. There are two subcases for this situation. 
\begin{itemize}
\item $i=j=k\neq l$.  Then
\beyy
M_{iiil} &=&\int u_i^3 u_l f(\bu) d\bu \\
              &= &((2\pi)^{-1/2}\int u_i^3 \exp\set{-\frac{1}{2}u_i^2}du_i)((2\pi)^{-1/2}\int u_l \exp\set{-\frac{1}{2}u_l^2}du_l) \\
              &=& 0
\eeyy
where $f_{\bu}(u)=(2\pi)^{-n/2}\exp\set{-\frac{1}{2}\bu^{\top}\bu}$ is the pdf of $\bu$.  On the other hand, we have
\[ A_{iiil} = \da_{ii}\da_{il}+  \da_{il}\da_{ii}+  \da_{ii}\da_{il}=0 \]
since $\da_{il}=0$ ($\da_{ij}$ is the Kronecker constant with $\da_{ij}=1$ iff $i=j$). This confirms $M_{iiil}=A_{iiil}$ in this subcase. \par
\item $i=j\neq k=l$.  Then
\beyy
M_{iikk} &=&\int u_i^2 u_k^2 f(\bu) d\bu \\
              &= &((2\pi)^{-1/2}\int u_i^2 \exp\set{-\frac{1}{2}u_i^2}du_i)((2\pi)^{-1/2}\int u_k^2 \exp\set{-\frac{1}{2}u_l^2}du_l) \\
              &=& E[u_i^2]E[u_k^2] =\sig_i^2 \sig_k^2 =1
\eeyy
On the other hand, we have
\[ A_{iikk} = \da_{ii}\da_{kk}+  \da_{ik}\da_{ik}+  \da_{ik}\da_{ik}=1  \]
This confirms $M_{iikk}=A_{iikk}$ in this subcase.
\end{itemize}
  
\noindent (3)  $t=3$.  We need to consider the case when $i=j$ and $i,k,l$ are distinct.  Then by the above argument we have
\[
M_{iikl} = E[u_i^2] E[u_k] E[u_l] =0
\]
for all distinct $i,k,l\in [n]$. On the other hand, we have
\[
A_{iikl} = \da_{ii}\da_{kl}+  \da_{ik}\da_{il}+  \da_{il}\da_{ik}=0
\]
Thus  $M_{iikl}=A_{iikl}$. \par
\noindent (4)  $t=4$, i.e., $i,j,k,l$ are all distinct.  Then similar argument as above yields $M_{ijkl}=A_{ijkl}$ for all distinct $(i,j,k,l)\in S(4,n)$.
 This concludes the proof that $m_4 = \A$.
\end{proof} 

\indent  The next two lemmas will be used to prove our main result.
\begin{lem}\label{le: u2au}
Let $\bu\in \R^n$ be a random vector and $\bv = A\bu$ where $A\in \R^{m\times n}$ is a constant matrix. Then  
\beq\label{eq: k-MbyTransf}
m_k(\bv) = [A]m_k(\bu)
\eeq
where $[A]m_k(\bu) = A\times_1\times_2\ldots \times_k m_k(\bu) $.  
\end{lem}

\begin{proof}
(\ref{eq: k-MbyTransf}) can be deduced by 
\[ m_k(\bv) = E[(A\bu)^k] = E[[A]\bu^k] =[A]E[\bu^k] = [A]m_k(\bu) . \] 
\end{proof}

\begin{lem}\label{le: mk_SND2ND}
Let $\bu\sim N_n(0,I_n)$ and $\bv = A\bu$ where $A\in \R^{m\times n}$ is constant. Then we have 
\begin{description}
\item[(1).]  For each odd  $k$,  $m_k(\bv)=0$.
\item[(2).]  For each even $k$, 
\beq\label{eq: MkTransf}
m_k(\bv) = \sum\limits_{\ga\in \pi_{k}} \Sig^{\ga}
\eeq
where $\Sig=AA^{\top}$ and $\pi_{k}$ is the set of 2-partitions of set $[k]$.     
\end{description}
\end{lem}

\begin{proof}
The first item can be shown by Lemma \ref{le: u2au} and $m_k(\bu)=0$ for odd $k$ due to (1) of Theorem \ref{th: M4sndvec}.  
To prove the second item, we note from Lemma \ref{le: u2au} that 
\beq\label{eq: mk2Mk}
m_k(\bv)=[A]E[\bu^k] = [A]\sum\limits_{\ga\in \pi_{k}} \cI_n^{\ga}
\eeq
The last equality is due to (2) of  Theorem \ref{th: M4sndvec}.  Furthermore,  
\beyy 
 [A] \cI_n^{\ga} &=& A\times_1\times_2\ldots \times_k (I_n\times_{\ga_2}I_n\times_{\ga_3}\ldots \times_{\ga_m} I_n)\\
                         &=&(AI_nA^{top}) \times_{\ga_2}(AI_nA^{top}) \times_{\ga_3}\ldots \times_{\ga_m}(AI_nA^{top})\\
                         &=&\Sig\times_{\ga_2}\Sig\times_{\ga_3}\ldots \times_{\ga_m}\Sig
\eeyy
It follows that  
\beq\label{eq: prf03}
[A] \cI_n^{\ga} = \Sig^{\ga}
\eeq
Consequently we get (\ref{eq: MkTransf}) by combining (\ref{eq: mk2Mk}) and (\ref{eq: prf03}). 
\end{proof}

\indent Now we are ready to express the high order moments and central moments of a general Gaussian vector in terms of its mean vector and covariance matrix. 

\begin{thm}\label{th: Moment4Gaussvec}
Let $\bx\in \R^n$ be a Gaussian vector with $\bx\sim N_n(\mu, \Sig)$ where $\mu\in \R^n$ and $\Sig\in \R^{n\times n}$ is a positive semidefinite matrix. Then
\beq\label{eq: moment4gvec}
m_k[\bx] = \sum\limits_{s=0}^{\lfix{k/2}}\sum\limits_{\bar{\ga}\in \Pi(s,k)} (\Sig\times_{\ga_1} \Sig\times_{\ga_2}\Sig\times\ldots \times_{\ga_s}\Sig\times_{\ga_{s+1}}\mu^{k-2s}) 
\eeq  
\end{thm} 

\begin{proof}
Since $\bx\sim N_n(0, \Sig)$.  We have $\bx\overset{d}{\sim}\mu +A\bu$, where $\bu\sim N_n(0,I_n)$ and $A\in \R^{n\times n}$ is symmetric with $A^2=\Sig$ ($A$ is a square root of $\Sig$). 
Then we have  
\bey\label{eq: prf01}
 m_k[\bx]  &=& E[\bx^k] =E[(\mu +A\bu)^k ] \\
                 &=& E[\sum\limits_{m=0}^k\sum\limits_{\theta\in \Phi(k-m,m)} (A\bu)^m\times_{\theta} \mu^{k-m}]\\
                 &=& \sum\limits_{m=0}^k\sum\limits_{\theta\in \Phi(k-m,m)} (E[A\bu)^m]\times_{\theta} \mu^{k-m}
\eey
where $\Phi(p,m):=\set{\theta:=(\theta_1,\theta_2,\ldots, \theta_p):  1\leq \theta_1<\theta_2<\ldots <\theta_p\leq m }$.  Since $E[(A\bu)^m] =0$ for each odd $m$ by Lemma \ref{le: mk_SND2ND}, 
we have, by the last equality of (\ref{eq: prf01})
\beyy
 m_k[\bx]  &=&  \sum\limits_{s=0}^{\lfix{k/2}}\sum\limits_{\theta\in \Phi(k-2s, 2s)} (E[(A\bu)^{2s}]\times_{\theta} \mu^{k-2s}\\
                 &=& \sum\limits_{s=0}^{\lfix{k/2}}\sum\limits_{\theta\in \Phi(k-2s, 2s)} ([A]E[\bu]^{2s}]\times_{\theta} \mu^{k-2s}\\
                 &=& \sum\limits_{s=0}^{\lfix{k/2}}\sum\limits_{\theta\in \Phi(k-2s, 2s)}  \sum\limits_{\ga\in \pi_{2s}} \Sig^{\ga} \times_{\theta} \mu^{k-2s}
  \eeyy
where $\ga=(\ga_1,\ga_2,\ldots, \ga_s)$ is a 2-partition of $[m]=[2s]$.  Set $\ga_{s+1}:=\theta\in \Phi(k-2s, 2s)$ and denote $\bar{\ga}=(\ga_1,\ga_2,\ldots, \ga_s,\ga_{s+1})$. Then 
$\bar{\ga}$ is a pseudo 2-partition of $[k]$.  (\ref{eq: moment4gvec}) is immediate. The proof is completed. 
\end{proof}

\indent  We end the paper by pointing out that the higher order moments of a Gaussian matrix can also be expressed similarly in tensor form, which may be investigated in our future work.  

%


\begin{thebibliography}{99}

\bibitem{Bjorn1996}
Bjorn, Holmquist,
Expectations of products of quadratic forms in normal variables.
Stochastic Analysis and Applications, 1996.

\bibitem{GN2000}
Ghazal G A , Neudecker H,
On second-order and fourth-order moments of jointly distributed random matrices: a survey,
Linear Algebra and Its Applications, 2000, 321(1-3):61-93.

\bibitem{Tracy1993}
Tracy D S , Sultan S A,
Higher order moments of multivariate normal distribution using matrix derivatives,
Stochastic Analysis and Applications, 1993, 11(3):337-348..

\bibitem{Bjorn1988}
Bjorn, Holmquist,
Moments and cumulants of the multivariate normal distribution,
Stochastic Analysis and Applications, 1988, 6(3):273-278.

\bibitem{Kendall1963}
Kendall M G,
The Advanced Theory of Statistics,
Revista Mexicana De Sociologa, 1963, 23(1):310.

\bibitem{JKB2000}
Johnson, N.L., Kotz, S., and Balakrishnan, N,
Continuous Multivariate Distributions, Vol1, 2nd. ed. Wiley, New York, 2000.

\bibitem{ZT2007}
Zhao X, Tan Z,
Application of CAPM with higher moments in insurance industry,
Journal of Shanxi University of Finance and Economics, 2007, 029(0z1): 106-107.

\bibitem{KR2005}
Tonu Kollo, Dietrich von Rosen,
Advanced Multivariate Statistics with Matrices, Springer, 2005, 10.1007/1-4020-3419-9.

\bibitem{ZHW2009}
Zhou J, Hao Z, and Wu Z,
Application of slice Wigner higher-order moment spectrum in fault diagnosis,
Journal of Liaoning Technical University (Natural Science Edition), 2009(3): 90-92.

\bibitem{WLD1988}
Wang Y, Li H, Dai C,
Application of parallel high-order moment method in simulation of large reflector antenna,
Journal of Microwave, 2016.

\end{thebibliography}
\end{document}